\documentclass[12pt]{amsart}

\usepackage{amssymb}
\usepackage{amsmath}
\usepackage{amsthm}
\usepackage{graphicx}
\usepackage{enumerate}
\usepackage{tikz}
\usetikzlibrary{arrows}
\usetikzlibrary{arrows.meta}
\tikzset{>={Latex[width=1.2mm,length=1.7mm]}}
\usepackage{tabularx}
\usepackage{mathrsfs}

\newtheorem{thm}{Theorem}[section]
\newtheorem{prop}[thm]{Proposition}
\newtheorem{cor}[thm]{Corollary}
\newtheorem{lem}[thm]{Lemma}

\newtheorem{conj}[thm]{Conjecture}
\newtheorem{prob}[thm]{Problem}
\theoremstyle{definition}
\newtheorem{alg}[thm]{Algorithm}

\newcommand{\sn}{\mathfrak{S}_n}
\newcommand{\sr}{\mathfrak{S}_r}
\newcommand{\mfs}[1]{\mathfrak{S}_{#1}}

\newcommand{\zx}{\mathbb{Z}[x]}

\newcommand{\zxn}{\mathbb{Z}[x_{1,1},x_{1,2},\dotsc,x_{n,n}]}

\newcommand{\tp}[1]{\mathcal M^{\textsc{TP}}_{#1}}
\newcommand{\tnn}[1]{\mathcal M^{\textsc{TNN}}_{#1}}
\newcommand{\hps}[1]{\mathcal M^{\textsc{HPS}}_{#1}}

\newcommand{\sumsb}[1]{\sum_{\substack{#1}}}  

\newcommand{\defeq}{:=}

\newcommand{\spn}{\mathrm{span}}

\newcommand{\ntnsp}{\negthinspace}
\newcommand{\ntksp}{\negthickspace}

\newcommand{\bp}{\begin{prob}}
\newcommand{\ep}{\end{prob}}

\newcommand{\mat}[1]{\mathrm{Mat}_{#1 \times #1}}

\newcommand{\perm}{\mathrm{per}}

\newcommand{\permmon}[2]{#1_{1,#2_1} \ntnsp\cdots #1_{n,#2_n}}

\newcommand{\ssm}{\smallsetminus}
\newcommand{\multiu}{\Cup}

\begin{document}
\author{Mark Skandera and Daniel Soskin}
\title{Permanental inequalities for totally positive matrices}

\bibliographystyle{dart}

\date{\today}

\begin{abstract}
We characterize ratios of permanents of (generalized) submatrices which are bounded 
on the set of all totally positive matrices.  This provides a permanental
analog of results of Fallat, Gekhtman, and Johnson
[{\em Adv.\ Appl.\ Math.} {\bf 30} no.\ 3, (2003) pp.\ 442--470]
concerning ratios of matrix minors.
We also extend work of Drake, Gerrish, and the first author [{\em Electron.\ J.\ Combin.,} {\bf 11} no.\ 1, (2004) Note 6] by characterizing the differences of monomials in $\zxn$ which evaluate positively on the set of all totally positive $n \times n$ matrices.
\end{abstract}

\maketitle

\section{Introduction}\label{s:intro}

Given an
$n\times n$ matrix $A = (a_{i,j})$
and subsets $I, J \subseteq [n] \defeq \{1,\dotsc, n\}$,
let $A_{I,J} = (a_{i,j})_{i \in I, j \in J}$ denote the $(I,J)$-{\em submatrix} of $A$.
For $|I|= |J|$, call $\det(A_{I,J})$ the $(I,J)$-{\em minor} of $A$.
A real $n\times n$ matrix $A$ is called \emph{totally positive} ({\em totally nonnegative}) if every minor of $A$ is positive (nonnegative). Let $\tp n \subset \tnn n$ denote these sets of matrices.

These and the set $\hps n$ of $n \times n$ Hermitian positive semidefinite matrices
arise in many areas of mathematics,
and for more than a century mathematicians have been studying inequalities satisfied by their matrix entries. 
(See, e.g., \cite{FJTNNMatrices}.)
Many such inequalities involve minors and permanents.
For instance
inequalities of Fischer~\cite{fischer1908}, 
Fan~\cite{carlson1967},
and Lieb~\cite{LiebPerm}
state that 
 for all matrices $A \in \tnn n \cup \hps n$, 
 and for all $I \subseteq [n]$
 and $I^c \defeq [n] \ssm I$, we have
    \begin{equation}\label{eq:fischer}
    \begin{gathered}
    \det(A)\leq \det(A_{I,I}) \det(A_{I^c,I^c}),\\
    \perm(A) \geq \perm(A_{I,I})\; \perm(A_{I^c,I^c}).
    \end{gathered}
    \end{equation}
Koteljanskii's inequality~\cite{KotelProp}, \cite{KotelPropRuss} states that for $A \in \tnn n \cup \hps n$ and for all $I, J \subseteq [n]$ 
we have
\begin{equation}\label{eq:koteljanskii}
    \det(A_{I\cup J, I\cup J})\det(A_{I\cap J, I\cap J})\leq \det(A_{I,I}) \det(A_{J,J}).
    \end{equation}

Many open questions about inequalities exist and seem difficult. For instance, it is known which $8$-tuples 
$(I,J,K,L,I',J',K',L')$ of subsets satisfy
\begin{equation}
    \det(A_{I,I'})\det(A_{J,J'}) \leq \det(A_{K,K'})\det(A_{L,L'}) 
\end{equation}
for all 
$A \in \tnn n$ \cite{FGJMult}, \cite{SkanIneq}, but few permanental analogs of such inequalities are known.  While some of these $8$-tuples also satisfiy
\begin{equation}
    \perm(A_{I,I'})\;\perm(A_{J,J'}) \geq \perm(A_{K,K'})\;\perm(A_{L,L'}),
\end{equation}    
    this second inequality is not true in general.
    For example, the 
 natural permanental analog
\begin{equation}\label{eq:kotelperm}
\perm(A_{I\cup J, I\cup J})\;\perm(A_{I\cap J, I\cap J})\geq \perm(A_{I,I})\; \perm(A_{J,J}),
\end{equation}
of (\ref{eq:koteljanskii}) 
holds neither for all
$A \in \hps n$ nor for all $A \in \tnn n$.
(See \cite[\S 6]{SkanSoskinBJArx} for a counterexample with $n = 3$.)

Let us put aside $\hps n$ and consider 
conjectured inequalities of the form
\begin{equation}\label{eq:productineq}
\mathrm{product}_1 \leq \mathrm{product}_2
\end{equation}
involving minors and permanents of matrices in $\tnn n$
and $\tp n$.
One strategy for studying (\ref{eq:productineq}) is to 
view the difference $\mathrm{product}_2 - \mathrm{product}_1$ as a polynomial 
\begin{equation}\label{eq:poly}
    f(x) \defeq f(x_{1,1}, x_{1,2}, \dotsc, x_{n,n}) \in \zx \defeq \zxn
\end{equation}
in matrix entries.  Then the validity of the inequality (\ref{eq:productineq}) is equivalent to the statement that for all $A = (a_{i,j}) \in \tnn n$, we have
\begin{equation}\label{eq:tnnpoly}
    f(A) \defeq f(a_{1,1}, a_{1,2}, \dotsc, a_{n,n}) \geq 0.
\end{equation}
We call a polynomial (\ref{eq:poly}) with this property
a {\em totally nonnegative polynomial}.
Since $\tp n$ is dense in $\tnn n$, the inequality (\ref{eq:tnnpoly}) holds for all $A \in \tp n$ if and only if it holds for all $A \in \tnn n$.

A second strategy 
for studying (variations of) a potential inequality (\ref{eq:productineq})
is to ask for which positive constants $k_1$, $k_2$
the modified inequalities
\begin{equation}\label{eq:modineq}
k_1 \cdot \mathrm{product}_1 \leq \mathrm{product}_2 
\leq k_2 \cdot \mathrm{product}_1
\end{equation}
hold for all $A \in \tnn n$.  
Bounds of $k_1 = 1$ or $k_2 = 1$ imply
the inequality (\ref{eq:productineq}) or its reverse to hold; 
other bounds give information not apparent in the proof or disproof of (\ref{eq:productineq}).
Equivalently, we may view the ratio of $\mathrm{product}_2$ to $\mathrm{product}_1$
as a rational function
\begin{equation}\label{eq:origratio}
    R(x) \defeq R(x_{1,1}, x_{1,2}, \dotsc, x_{n,n}) \in \mathbb Q(x) \defeq \mathbb Q(x_{1,1}, x_{1,2}, \dotsc, x_{n,n})
\end{equation}
in matrix entries, and we may ask for upper and lower bounds
as $x$ varies over $\tp n$.
While a ratio (\ref{eq:origratio})  
is not defined everywhere on $\tnn n$, 
the density of $\tp n$ in $\tnn n$ allows us to
restrict our attention to $\tp n$:
we have
\begin{equation}\label{eq:bounds}
k_1 \leq R(x) \leq k_2
\end{equation}
for all $x \in \tp n$ if and only if
the same inequalities hold for all 
$x \in \tnn n$ such that $R(x)$ is defined.
Clearly the lower bound $k_1$ is interesting
only when positive,
since products of minors and permanents of totally nonnegative matrices are trivially bounded below by $0$.




A characterization of all ratios of the form
\begin{equation}\label{eq:2over2}
\frac{\det(x_{I,I'})\det(x_{J,J'})}
{\det(x_{K,K'})\det(x_{L,L'})}, \qquad I, I',\dotsc, L, L' \subset [n],
\end{equation}
which are bounded above and/or nontrivially bounded below on $\tp n$
follows from work in
\cite{FGJMult} and \cite{SkanIneq}.
Each ratio (\ref{eq:2over2}) is bounded above and/or below by $1$, and for each $n$, 
factors as a product of elements of a finite set of indecomposable ratios.  This result was extended in \cite{GSBoundedRatios} to include ratios of products of arbitrarily many minors
\begin{equation}\label{eq:manyovermany}
\frac{\det(x_{I_1,I'_1}) \cdots \det(x_{I_p,I'_p})}
{\det(x_{J_1,J'_1})\cdots\det(x_{J_p,J'_p})}.
\end{equation}
Again, each of these factors
as a product of elements belonging to a finite set of indecomposable ratios.  For $n = 3$, each ratio (\ref{eq:manyovermany}) is bounded above and/or below by $1$;
for $n \geq 4$, such bounds are conjectured~\cite{BoochGen}.



 While the permanental version (\ref{eq:kotelperm}) of Koteljanskii's inequality is false, we will show in Section~\ref{s:main} that the corresponding ratio is bounded above and nontrivially below.
 Specifically,
\begin{equation}\label{eq:kotelratio}
\frac1{|I \cup J|! ~| I \cap J|!}
\leq
\frac{\perm(x_{I,I})\;\perm(x_{J,J})}
{\perm(x_{I\cup J, I\cup J})\;\perm(x_{I\cap J, I\cap J})}
\leq
|I|!~|J|!
\end{equation}
for all $I, J \subseteq [n]$ and $x \in \tp n$.
The 
failure of (\ref{eq:kotelperm}), combined with (\ref{eq:kotelratio}),
exposes a difference between ratios of minors and of permanents:
unlike the bounded ratios in (\ref{eq:2over2}),
{\em not} all bounded ratios of permanents are bounded by $1$.  Thus it is natural to ask
which ratios 
\begin{equation}\label{eq:R}
R(x)=\frac{\text{per}(x_{I_1,I'_1})\text{per}(x_{I_2,I'_2})\cdots\text{per}(x_{I_r,I'_r})}{\text{per}(x_{J_1,J'_1})\text{per}(x_{J_2,J'_2})\cdots\text{per}(x_{J_q,J'_q})}
\end{equation} 
are bounded above and/or nontrivially below as real-valued functions on $\tp n$,
and to state bounds.

In Section~\ref{s:multitnnorder} we describe
a multigrading of 
the coordinate ring $\zx$ of $n \times n$ matrices.  Extending work in \cite{DGSBruhat}, we define a partial order on the monomials in 
$\zx$ which characterizes the differences 
$\smash{\prod x_{i,j}^{c_{i,j}} - \prod x_{i,j}^{d_{i,j}}}$ 
which are totally nonnegative polynomials.
This leads to our main results in Section~\ref{s:main} which characterize ratios (\ref{eq:R}) which are bounded above and nontrivially below as real-valued functions on $\tp n$.  We provide some such bounds, which are not necessarily tight.
We finish in Section~\ref{s:open} with some open questions.
\section{A multigrading of $\zxn$ and the total nonnegativity order}\label{s:multitnnorder}


We will find it convenient to view degree-$r$ monomials in $\zx$
in terms of permutations in the symmetric group $\sr$ and multisets of
$[n]$.
In particular, given permutations $v, w \in \sr$ define
the monomial
\begin{equation*}
    x^{v,w} \defeq x_{v_1,w_1} \cdots x_{v_r,w_r}.
\end{equation*}
Define an {\em $r$-element multiset} of $[n]$ to be a nondecreasing
$r$-tuple of elements of $[n]$.
In exponential notation, we write $i^{k}$ to represent $k$ consecutive
occurrences of $i$ in such an $r$-tuple, e.g.,
\begin{equation}\label{eq:multisetex}
    (1,1,2,3) = 1^22^13^1, \qquad (1,2,2,2) = 1^12^3.
\end{equation}
Two $r$-element multisets
\begin{equation}\label{eq:MO}
M = (m_1,\dotsc,m_r) = 1^{\alpha_1} \cdots n^{\alpha_n}, \qquad
O = (o_1,\dotsc,o_r) = 1^{\beta_1} \cdots n^{\beta_n},
\end{equation}
determine a {\em generalized submatrix}
$x_{M,O}$ of $x$ by $(x_{M,O})_{i,j} \defeq x_{m_i,o_j}$.
For example, when $n=3$, we have the $4 \times 4$ generalized submatrix
and monomial
\begin{equation}\label{eq:gensubmatrixex}
    x_{1123,1222} = \begin{bmatrix}
        x_{1,1} & x_{1,2} & x_{1,2} & x_{1,2} \\
        x_{1,1} & x_{1,2} & x_{1,2} & x_{1,2} \\
        x_{2,1} & x_{2,2} & x_{2,2} & x_{2,2} \\
        x_{3,1} & x_{3,2} & x_{3,2} & x_{3,2}
    \end{bmatrix}\ntksp,
    \qquad
    (x_{1123,1222})^{1234,4312} = x_{1,2} x_{1,2} x_{2,1} x_{3,2}.
\end{equation}

The ring $\zx$ has a natural multigrading
\begin{equation}\label{eq:multigrading}
  \mathbb Z[x] = \bigoplus_{r \geq 0} \bigoplus_{M,O} \mathcal A_{M,O},
\end{equation}
where the second direct sum is over pairs $(M,O)$ of $r$-element
multisets of $[n]$, and
\begin{equation}\label{eq:permmon}
  \mathcal A_{M,O} \defeq
  \spn_{\mathbb Z}\{ (x_{M,O})^{e,w} \,|\, w \in \sr \}.
  \end{equation}
More precisely, for $M$, $O$ as in (\ref{eq:MO}),
a basis for $\mathcal A_{M,O}$ is given by all monomials
\begin{equation}\label{eq:multimon}
  \prod_{i,j \in [n]} \ntksp x_{i,j}^{c_{i,j}}
  \end{equation}
with $C = (c_{i,j}) \in \mathrm{Mat}_{n\times n}(\mathbb N)$ satisfying
\begin{equation}\label{eq:CtoMO}
    c_{i,1} + \cdots + c_{i,n} = \alpha_i,
    \quad
    c_{1,j} + \cdots + c_{n,j} = \beta_j
    \quad \text{for } i,j=1,\dotsc,n.
  \end{equation}
One may express a monomial (\ref{eq:multimon})
in the form $(x_{M,O})^{e,w}$
by the following algorithm.
\begin{alg}\label{a:mattoperm}
  Given a monomial (\ref{eq:multimon})
  in $\mathcal A_{M,O}$
  with $M$, $O$ as in (\ref{eq:MO}),
  \begin{enumerate}[(i)]
    \item Define the rearrangement $u = u_1 \cdots u_r$ of $O$ by
  writing (\ref{eq:multimon}) with variables in lexicographic order
  as $x_{m_1,u_1} \cdots x_{m_r,u_r}$.
\item Let $j_1 < \cdots < j_{\beta_1}$ be
  the positions of the $\beta_1$ ones in $u$,
  let $j_{\beta_1 + 1} < \cdots < j_{\beta_1 + \beta_2}$ be
  the positions of the $\beta_2$ twos in $u$, etc.
\item For $i = 1, \dotsc, r$, define $w_{j_i} = i$.
\item Call the resulting word $w = w(C)$.
\end{enumerate}
\end{alg}
For example, it is easy to check that for $n=3$ and multisets
$(1123,1222) = (1^22^13^1, 1^12^33^0)$ of $\{1,2,3\}$,
the graded component $\mathcal A_{1123,1222}$
of $\mathbb Z[x_{1,1},x_{1,2},\dotsc,x_{3,3}]$
is spanned by monomials (\ref{eq:multimon}),
where $C = (c_{i,j})$ is one of the matrices
\begin{equation}\label{eq:3x3matrices}
  \begin{bmatrix}
    1 & 1 & 0 \\ 0 & 1 & 0 \\ 0 & 1 & 0
  \end{bmatrix}\ntksp,
  \quad
  \begin{bmatrix}
    0 & 2 & 0 \\ 1 & 0 & 0 \\ 0 & 1 & 0
  \end{bmatrix}\ntksp,
  \quad
  \begin{bmatrix}
    0 & 2 & 0 \\ 0 & 1 & 0 \\ 1 & 0 & 0
  \end{bmatrix}
\end{equation}
having row sums $(2,1,1)$ and column sums $(1,3,0)$.  These are
\begin{equation}\label{eq:3monomials}
   x_{1,1}x_{1,2}x_{2,2}x_{3,2}, \quad
   x_{\smash{1,2}}^2x_{2,1}x_{3,2}, \quad
   \smash{x_{1,2}^2}x_{2,2}x_{3,1},
\end{equation}
with column index sequences equal to the rearrangements
$1222$, $2212$, $2221$ of $1222$.
Algorithm~\ref{a:mattoperm} then produces
permutations $1234$, $2314$, $2341$ in $\mfs 4$, and we may express
the monomials (\ref{eq:3monomials}) as
\begin{equation}\label{eq:3moremonomials}
  (x_{1123,1222})^{1234,1234},\quad
  (x_{1123,1222})^{1234,2314},\quad
  (x_{1123,1222})^{1234,2341}.
\end{equation}


For $r$-element multisets $M$, $O$ of $[n]$,
the monomials
in $\mathcal A_{M,O}$ are closely
related to parabolic subgroups of $\mfs r$ with standard generators
$s_1, \dotsc, s_{r-1}$, and double cosets of the form
$W_{\iota(M)} w W_{\iota(O)}$
where $w$ belongs to
$\mfs r$,
$W_J$ is the subgroup of $\mfs r$ generated by $J$, and
  \begin{equation}\label{eq:parabolics}
    \begin{gathered}
      \iota(M) \defeq \{ s_1, \dotsc, s_{r-1} \} \ssm
      \{s_{\alpha_1}, s_{\alpha_1 + \alpha_2}, \dotsc, s_{r-\alpha_n} \}
      = \{ s_j \,|\, m_j = m_{j+1} \},\\
      \iota(O) \defeq \{ s_1, \dotsc, s_{r-1} \} \ssm
      \{s_{\beta_1}, s_{\beta_1 + \beta_2}, \dotsc, s_{r-\beta_n} \}
      = \{ s_j \,|\, o_j = o_{j+1} \}.
    \end{gathered}
  \end{equation}
  It is easy to see that the map $M \mapsto \iota(M)$ is bijective:
  one recovers $M = 1^{\alpha_1} \cdots n^{\alpha_n}$ from the generators not in $\iota(M)$
  as in (\ref{eq:parabolics}).
  It is known that each double coset has unique minimal and maximal elements
  with respect to the {\em Bruhat order on $\mfs r$}, defined by declaring $v \leq w$ if each reduced expression $s_{i_1} \cdots s_{i_\ell}$ for $w$ contains a subword which is a reduced expression for $v$.  (See, e.g., \cite{BBCoxeter}, \cite{DougInv}.)
    Let $W_{\iota(M)} \backslash W / W_{\iota(O)}$ denote
the set of all double cosets of $W = \sr$
determined by $r$-element multisets $M$, $O$.
\begin{prop}\label{p:MOcosetcorr}
  Fix $r$-element multisets $M$, $O$ as in (\ref{eq:MO}).
    The double cosets
$W_{\iota(M)} \backslash W / W_{\iota(O)}$ satisfy the following.
\begin{enumerate}[(i)]
\item Each double coset
  has a unique Bruhat-minimal element $u$
  satisfying $su > u$ for all $s \in \iota(M)$ and $us > u$ for all
  $s \in \iota(O)$;
  it has a unique Bruhat-maximal element $u'$
  satisfying $su' < u'$ for all $s \in \iota(M)$ and $u's < u'$ for all
  $s \in \iota(O)$.
\item We have $W_{\iota(M)}vW_{\iota(O)} = W_{\iota(M)}wW_{\iota(O)}$
      if and only if $(x_{M,O})^{e,v} = (x_{M,O})^{e,w}$.
\item The cardinality $|W_{\iota(M)} \backslash W / W_{\iota(O)}|$
  is the dimension of $\mathcal A_{M,O}$, equivalently, the number
  of matrices in $\mat n(\mathbb N)$ having row sums $(\alpha_1,\dotsc, \alpha_n)$
  and column sums $(\beta_1,\dotsc,\beta_n)$.
\item Each permutation $w$ produced by
Algorithm~\ref{a:mattoperm} is the unique Bruhat-minimal element of
its coset $W_{\iota(M)}wW_{\iota(O)}$.
\end{enumerate}
\end{prop}
  \begin{proof}
    (i) See \cite{DougInv}.

    \noindent (ii) The dimension of $\mathcal A_{M,O}$ is the cardinality
    of
    the set
    $\{ (x_{M,O})^{e,w} \,|\, w \in \sr \}$.
    But we have $(x_{M,O})^{e,v} = (x_{M,O})^{e,w}$
    if and only if when we partition the $r \times r$
    permutation matrices $P(v)$, $P(w)$ of $v$, $w$ into blocks by drawing bars
    after rows $\alpha_1, \alpha_1+\alpha_2,\dotsc,r-\alpha_n$ and after columns
    $\beta_1, \beta_1+\beta_2, \dotsc, r-\beta_n$, the corresponding blocks of $P(v)$ and $P(w)$
    contain equal numbers of ones.
    It follows that for fixed $w \in \sr$, the set
    $\{ v \in \sr \,|\, (x_{M,O})^{e,v} = (x_{M,O})^{e,w} \}$
    is $W_{\iota(M)}wW_{\iota(O)}$.

    \noindent (iii)  This follows from (ii), where
    $c_{i,j}$ is the number of ones in block $(i,j)$ of the permutation
    matrix of any permutation belonging to the double coset.

    \noindent (iv) By Step (i) of the algorithm,
    subwords $w_{1} \cdots w_{\alpha_1}$, $w_{\alpha_1+1}\cdots w_{\alpha_1+\alpha_2}$,
    etc., of $w(C)$ are increasing.
    It follows that for any
    generator $s \in \iota(M)$ we have $sw > w$.
    By Step (ii) of the algorithm,
    letters $1,\dotsc,\beta_1$ appear in increasing order in $w(C)$,
    as do $\beta_1+1,\dotsc,\beta_1+\beta_2$, etc.
    It follows that for any
    generator $w \in \iota(O)$ we have $ws > w$.
  \end{proof}

  For any subsets $I$, $J$ of generators of $\sr$,
  the Bruhat order on $\mfs r$ induces a poset structure on
  $W_I \backslash W / W_J$
  as follows.  We declare
$W_IvW_J \leq W_IwW_J$ if elements of the cosets satisfy any of the
three (equivalent) inequalities in the Bruhat order on $\mfs r$~\cite[Lem.\,2.2]{DougInv}.
  \begin{enumerate}[(i)]
      \item The minimal element of $W_IvW_J$ is less than or equal to
the minimal element of $W_IwW_J$.
      \item The maximal element of $W_IvW_J$ is less than or equal to
the maximal element of $W_IwW_J$.
      \item At least one element of $W_IvW_J$ is less than or equal to
at least one element of $W_IwW_J$.
  \end{enumerate}
  We call this poset the {\em Bruhat order on
  $W_I \backslash W / W_J$}.
  A fourth equivalent inequality can be stated in terms of
  matrices of exponents defined by monomials in $\mathcal A_{M,O}$.
  (See, e.g., \cite{HoSkanDoubleCoset}.)
  Given a matrix $C = (c_{i,j}) \in \mat n(\mathbb N)$, 
  define the matrix
  $C^* = (c^*_{i,j}) \in \mat n(\mathbb N)$
  by
  \begin{equation}\label{eq:cstar}
    c^*_{i,j} =
    \text{sum of entries of $C_{[i],[j]}$}.
  \end{equation}
  \begin{prop}\label{p:doubleehresmann}
    Fix monomials
    \begin{equation*}
      (x_{M,O})^{e,v} = \prod_{i,j} x_{i,j}^{c_{i,j}},
      \quad
      (x_{M,O})^{e,w} = \prod_{i,j} x_{i,j}^{d_{i,j}},
    \end{equation*}
    in $\mathcal A_{M,O}$
    and define matrices $C^*$, $D^*$ as in (\ref{eq:cstar}).
    Then we have
    $W_{\iota(M)}vW_{\iota(O)} \leq W_{\iota(M)}wW_{\iota(O)}$
    in the Bruhat order if and only if $C^* \geq D^*$
    in the componentwise order.
  \end{prop}

  The Bruhat order on $W_{\iota(M)}\backslash W/W_{\iota(O)}$
    is closely related to certain totally nonnegative polynomials
    in $\mathcal A_{M,O}$.
%
    Indeed, when $M = O = 1^n$, totally nonnegative polynomials of the form
    $x^{e,v} - x^{e,w}$
    are characterized by the Bruhat order on $\sn$~\cite{DGSBruhat}.
  \begin{thm}\label{t:DGS}
    For $v, w \in \sn$,
    the polynomial $x^{e,v} - x^{e,w}$ is totally nonnegative
    if and only if $v \leq w$ in the Bruhat order.
  \end{thm}

  We will now extend this result to all monomials in $\zx$.
  Let us define a partial order $\leq_T$
  on all monomials in $\zx$ by declaring
    $(x_{M,O})^{e,v} \leq_T (x_{P,Q})^{e,w}$  if 
    $(x_{P,Q})^{e,w} - (x_{M,O})^{e,v}$ is a totally nonnegative polynomial.
We call this the {\em total nonnegativity order} on monomials in $\zx$.
It is not hard to show that
the total nonnegativity order is a disjoint
union of its restrictions to the multigraded components (\ref{eq:multigrading})
of $\zx$.
\begin{lem}
  Monomials
\begin{equation}\label{eq:monomials}
  \prod_{i,j} \ntnsp x_{i,j}^{c_{i,j}},
  \quad
  \prod_{i,j} \ntnsp x_{i,j}^{d_{i,j}}
\end{equation}
are comparable in the total nonnegativity order
only if they belong to the same multigraded component
of $\mathbb Z[x]$.
\end{lem}
\begin{proof}
  For $k,\ell \in [n]$ and
  $t \in \mathbb R_{\geq 0}$,
  define the $n \times n$ matrix $E^{k,\ell}(t) = (e_{i,j}^{k,\ell})_{i,j \in [n]}$ by
  \begin{equation}\label{eq:eklt}
    e_{i,j}^{k,\ell} = \begin{cases}
      t &\text{if $i \leq k$ and $j \leq \ell$},\\
      1 &\text{otherwise}.
      \end{cases}
      \end{equation}
This matrix is totally nonnegative if $t \geq 1$, or $k = n$, or $\ell = n$.
  

Suppose that the monomials belong to
components $\mathcal A_{M,O}$ and $\mathcal A_{M',O'}$ of $\zx$,
with $M$, $O$, as in (\ref{eq:MO}) and
\begin{equation*}
  M' = 1^{\alpha_1'} \cdots n^{\alpha_n'}, \qquad O' = 1^{\beta_1'} \cdots n^{\beta_n'}.
\end{equation*}
If $M \neq M'$, then let $k \in [n]$ be the least index
appearing with different multiplicities in the two multisets.
Evaluating the monomials (\ref{eq:monomials}) at $E^{k,n}(t)$ yields
$t^{\alpha_1 + \cdots + \alpha_k}$ and $t^{\alpha'_1 + \cdots + \alpha'_k}$.  The difference of
these can be made positive or negative by choosing $t<1$ or $t>1$.
On the other hand, if $O \neq O'$, then
the evalulation of the monomials (\ref{eq:monomials}) at matrices
of the form $E^{n,\ell}(t)$ leads to a similar conclusion.
\end{proof}



\begin{thm}
  Fix $r$-element multisets $M = 1^{\alpha_1} \cdots n^{\alpha_n}$, $O = 1^{\beta_1} \cdots n^{\beta_n}$ 
  as in (\ref{eq:MO}),
  and matrices $C, D \in \mat n(\mathbb N)$ with row and column sums
  $(\alpha_1,\dotsc,\alpha_n)$, $(\beta_1,\dotsc,\beta_n)$, and define the
  polynomial
    \begin{equation*}
    f(x) = \prod_{i,j}\ntnsp x_{i,j}^{c_{i,j}} - \prod_{i,j}\ntnsp x_{i,j}^{d_{i,j}}
  \end{equation*}
in $\mathcal A_{M,O}$.
  Then the following are equivalent.
    \begin{enumerate}[(i)]
      \item $f(x)$
  is totally nonnegative.
\item $C^* \geq D^*$ in the componentwise order.
\item $w(C) \leq w(D)$ in the Bruhat order on $\sr$.
\item $f(x)$ is equal to a sum of products
of the form $\det(x_{I,J})x_{u_1,v_1} \cdots x_{u_{r-2},v_{r-2}}$ in $\mathcal A_{M,O}$
with $|I| = |J| = 2$.
  \end{enumerate}
  \end{thm}
  \begin{proof}
    (i $\Rightarrow$ ii)
    Suppose $C^* \not \geq D^*$ in the componentwise order and
    let $(k,\ell)$ the the lexicographically least pair
    satisfying $c^*_{k,\ell} < d^*_{k,\ell}$.
    Now choose $t > 1$ and evaluate $f(x)$ at the totally nonnegative
    matrix $E^{k,\ell}(t)$ to obtain $t^{c^*_{k,\ell}} - t^{d^*_{k,\ell}} < 0$.
    It follows that $f(x)$ is not a totally nonnegative polynomial.
    
    \noindent
    (ii $\Rightarrow$ iii)  
    This follows from Proposition~\ref{p:doubleehresmann}
    and the definition of the Bruhat order on double cosets.

    \noindent
    (iii $\Rightarrow$ iv)
    Suppose that $w(C) \leq w(D)$ and let $p = \ell(w(D)) - \ell(w(C))$. Then there exist 
    a sequence 
    \begin{equation*}
        w(C) = y^{(0)} < y^{(1)} < \cdots < y^{(p-1)} < y^{(p)} = w(D)
    \end{equation*}
    of permutations and a
    sequence $( (i_0,j_0), \dotsc, (i_{p-1\ntnsp},j_{p-1}))$
    of transpositions in $\sr$ such that
    we have
    \begin{equation*}
        y^{(k)} = (i_{k-1\ntnsp},j_{k-1})y^{(k-1)}, \qquad \ell(y^{(k)}) = \ell(y^{(k-1)}) + 1
    \end{equation*}
    for $k = 1, \dotsc, p$.  We may thus write $f(x) = (x_{M,O})^{e,w(C)} - (x_{M,O})^{e,w(D)}$ as the telescoping sum
    \begin{equation*}
        \Big(\ntnsp(x_{M,O})^{e,y^{(0)}} \ntksp- (x_{M,O})^{e,y^{(1)}}\Big)
        + \Big(\ntnsp(x_{M,O})^{e,y^{(1)}} \ntksp- (x_{M,O})^{e,y^{(2)}}\Big)
        + \cdots +
        \Big(\ntnsp(x_{M,O})^{e,y^{(p-1)}} \ntksp- (x_{M,O})^{e,y^{(p)}}\Big),
    \end{equation*}
    where each parenthesized difference either has the desired
    form 
    or is $0$.

\noindent 
(iv $\Rightarrow$ i)
A sum of products of minors is a totally nonnegative polynomial.
 \end{proof}


  For example, let us revisit the monomials
  (\ref{eq:3monomials}) -- (\ref{eq:3moremonomials})
  in the graded component
  $\mathcal A_{1123,1222}$ of $\mathbb Z[x_{1,1}, x_{1,2},\dotsc,x_{3,3}]$.
    It is easy to see that $1234 < 2314 < 2341$ in the Bruhat order on
    $\mfs 4$ and that the application of (\ref{eq:cstar})
    to the corresponding matrices in (\ref{eq:3x3matrices})
    yields the componentwise comparisons
  \begin{equation}\label{eq:star3x3matrices}
  \begin{bmatrix}
    1 & 2 & 2 \\ 1 & 3 & 3 \\ 1 & 4 & 4
  \end{bmatrix}
  \geq
  \begin{bmatrix}
    0 & 2 & 2 \\ 1 & 3 & 3 \\ 1 & 4 & 4
  \end{bmatrix}
  \geq
  \begin{bmatrix}
    0 & 2 & 2 \\ 0 & 3 & 3 \\ 1 & 4 & 4
  \end{bmatrix} \ntksp.
\end{equation}
  Thus we have $(x_{1123,1222})^{1234,1234} \geq_T(x_{1123,1222})^{1234,2314} \geq_T (x_{1123,1222})^{1234,2341}$, i.e.,
\begin{equation*}
  x_{1,1}x_{1,2}x_{2,2}x_{3,2}
  \ \geq_T\
  x_{\smash{1,2}}^2x_{2,1}x_{3,2}
  \ \geq_T\
  \smash{x_{1,2}^2}x_{2,2}x_{3,1}.
\end{equation*}
Furthermore, the chain $1234 < 2134 < 2314 < 2341$ 
in the Bruhat order on $\mfs 4$
with 
\begin{equation}
2134 = (1,2)1234, \quad
2314 = (2,3)2134, \quad
2341 = (3,4)2314
\end{equation}
allows us to write
$x_{1,1}x_{1,2}x_{2,2}x_{3,2} -  x_{\smash{1,2}}^2x_{2,1}x_{3,2}$
as
\begin{equation*}
\begin{gathered}
    \Big( \ntnsp(x_{1123,1222})^{1234,1234} - (x_{1123,1222})^{1234,2134}\Big)
\,+\,
\Big(\ntnsp(x_{1123,1222})^{1234,2134} - (x_{1123,1222})^{1234,2314}\Big)\\
= \det \ntnsp\begin{bmatrix} x_{1,1} & x_{1,2} \\ x_{1,1} & x_{2,2} \end{bmatrix}\ntnsp x_{2,2} x_{3,2} 
\,+\,
x_{1,2}\det \ntnsp \begin{bmatrix} x_{1,1} & x_{1,2} \\ x_{2,1} & x_{2,2} \end{bmatrix}\ntnsp x_{3,2}, 
\end{gathered}
\end{equation*}
and to write 
 $x_{\smash{1,2}}^2x_{2,1}x_{3,2}
  -
  \smash{x_{1,2}^2}x_{2,2}x_{3,1}$
as
\begin{equation*}
\Big(\ntnsp(x_{1123,1222})^{1234,2314} - (x_{1123,1222})^{1234,2341}\Big)
= x_{1,2}^2 \det \ntnsp\begin{bmatrix} x_{2,1} & x_{2,2} \\ x_{3,1} & x_{3,2} \end{bmatrix}\ntnsp.
\end{equation*}

\section{Main results}\label{s:main}

Let $\tp n$ be the set of totally positive $n \times n$ matrices.
To characterize ratios of products of permanents which are bounded above and/or nontrivially bounded below on the set $\tp n$, we first consider necessary conditions on the multisets of rows and columns appearing in such ratios.
Let
\begin{equation}\label{eq:mainratio}
    R(x) = \frac{\perm(x_{I_1,I'_1})\perm(x_{I_2,I'_2})\cdots\perm(x_{I_r,I'_r})}{\perm(x_{J_1,J'_1})\perm(x_{J_2,J'_2})\cdots\perm(x_{J_q,J'_q})},
\end{equation} 
be such a ratio, where 
\begin{equation}\label{eq:multisetseqs}
(I_1, \dotsc, I_r), \quad
(I'_1,\dotsc,I'_r), \quad
(J_1,\dotsc,J_q), \quad
(J'_1,\dotsc,J'_q)
\end{equation}
are multisets of $[n]$ satisfying $|I_k| = |I'_k|$,
$|J_k| = |J'_k|$ for all $k$.
In order for
$R(x)$ to be bounded above or nontrivially bounded below on $\tp n$ the multisets (\ref{eq:multisetseqs}) must
be related in terms of
an operation which we call
{\em multiset union}.
Given multisets $M = 1^{\alpha_1} \cdots n^{\alpha_n}$, $O = 1^{\beta_1} \cdots n^{\beta_n}$ of $[n]$, define their multiset union to be
\begin{equation}\label{eq:multiuniondef}
M \multiu O \defeq 1^{\alpha_1 + \beta_1} \cdots n^{\alpha_n + \beta_n}.
\end{equation}
For example, $1124 \multiu 1233 = 11122334$.


\begin{prop}\label{p:ST0necc}
Given multiset sequences as in (\ref{eq:multisetseqs}),
a ratio 
(\ref{eq:mainratio})   
can be bounded above or nontrivially bounded below on $\tp n$ 
only if we have the multiset equalities
\begin{equation}\label{eq:ST0}
    I_1 \Cup \cdots \Cup I_r = J_1 \Cup \cdots \Cup J_q,
    \qquad
    I'_1 \Cup \cdots \Cup I'_r = J'_1 \Cup \cdots \Cup J'_q.
\end{equation}
\end{prop}
\begin{proof} Given a multiset $K$, let $\mu_i(K)$ denote the multiplicity of $i$ in $K$, and define
\begin{equation}
    \alpha_i = \sum_{k=1}^r \mu_i(I_k),
    \quad
     \beta_i = \sum_{k=1}^r \mu_i(I'_k),
    \quad
     \alpha'_i = \sum_{k=1}^q \mu_i(J_k),
    \quad
     \beta'_i = \sum_{k=1}^q \mu_i(J'_k).
\end{equation}
Assume that the multiset equalities (\ref{eq:ST0}) do not hold, e.g., for
some $i$ we have $\alpha_i \neq \alpha'_i$. Let $A$ be a totally positive matrix and construct a family of matrices $(A(t))_{t > 0}$ by scaling row $i$ of $A$ by $t$.
Clearly, each matrix $A(t)$ is totally positive, since each minor $\det(A(t)_{I,J})$ equals either $\det(A_{I,J})$ or $t$ times this.
Furthermore, we have $R(A(t))=t^{\alpha_i-\alpha'_i}R(A)$, since each permanent with row multiset $K$ containing $i$ is scaled by $t^{\mu_i(K)}$. 
Thus, 
by letting $t$ approach $0$ or $+\infty$, we can make $R(A(t))$ arbitrarily large or arbitrarily close to $0$.
The same is true if we have
$\beta_i \neq \beta'_i$.
\end{proof}

To state sufficient conditions for the boundedness of
ratios (\ref{eq:mainratio}) 
we observe that it is possible to bound
the permanent above and below as follows.
\begin{prop}\label{p:perminmax}
     For any $n \times n$ totally nonnegative matrix $A = (a_{i,j})$ we have
     \begin{equation}
        a_{1,1} \cdots a_{n,n} \leq \perm(A) \leq n!\cdot a_{1,1} \cdots a_{n,n}. 
    \end{equation}
\end{prop}
\begin{proof}
    The first inequality follows from the fact that $\permmon aw > 0$ for all $w \in \sn$.  The second inequality follows from the fact (Theorem~\ref{t:DGS}) that
    $a_{1,1} \cdots a_{n,n} \leq \permmon aw$ for all $w \in \sn$.
\end{proof}







Now we state our main result, which characterizes ratios $R(x)$ as in (\ref{eq:mainratio}) which are bounded above for $x \in \tp n$.
\begin{thm}\label{t:main}
Let rational function
\begin{equation}\label{eq:mainratiox}
    R(x) = \frac{\perm(x_{I_1,I'_1})\perm(x_{I_2,I'_2})\cdots\perm(x_{I_r,I'_r})}{\perm(x_{J_1,J'_1})\perm(x_{J_2,J'_2})\cdots\perm(x_{J_q,J'_q})}
\end{equation}    
have index sets which satisfy (\ref{eq:ST0}), and define
matrices $C = (c_{i,j})$, $C^* = (c^*_{i,j})$, $D = (d_{i,j})$, $D^* = (d^*_{i,j})$ by
\begin{equation}\label{eq:CandD}
(x_{I_1,I'_1})^{e,e} \cdots (x_{I_r,I'_r})^{e,e} = \prod x_{i,j}^{c_{i,j}},
\qquad
(x_{J_1,J'_1})^{e,e} \cdots (x_{J_q,J'_q})^{e,e} = \prod x_{i,j}^{d_{i,j}},
\end{equation}
and (\ref{eq:cstar}).
Then $R(x)$ is bounded above
on the set of totally positive matrices
if and only if 
$C^* \leq D^*$
in the componentwise order.
In this case, it is bounded above by
$|I_1|! \cdots |I_r|!\;$.
\end{thm}
\begin{proof}
Suppose that $C^* \nleq D^*$.
Then for some indices $(k,\ell)$ we have $c^*_{k,\ell} > d^*_{k,\ell}$.
Define the matrix $B(t) = (b_{i,j}(t))$ by
\begin{equation*}
    b_{i,j}(t) = \begin{cases}
        t &\text{if $i \leq k$ and $j \leq \ell$},\\
        1 &\text{otherwise}.
    \end{cases}
\end{equation*}   
Now, we have $R(B(t))=\frac{p(t)}{q(t)}$ where $\deg (p(t))=c^*_{i,j}>d^*_{i,j}=\deg (q(t))$.
Thus we have \begin{equation*}
    \lim_{t \rightarrow \infty} R(B(t)) = t^{c^*_{i,j}-d^*_{i,j}}= \infty.
    \end{equation*}


Assume therefore that we have $C^* \leq D^*$ and let $A$ be any $n \times n$ totally positive matrix.
Applying the inequalities of 
Proposition~\ref{p:perminmax} to the numerator and denominator of $R(A)$
respectively, we see that $R(A)$ is at most
\begin{equation}
\frac{|I_1|!(A_{I_1,I'_1})^{e,e} \cdots |I_r|!(A_{I_r,I'_r})^{e,e}}{(A_{J_1,J'_1})^{e,e} \cdots (A_{J_q,J'_q})^{e,e}} 
    = \frac{|I_1|! \cdots |I_r|! \prod a_{i,j}^{c_{i,j}}}{\prod a_{i,j}^{d_{i,j}}}.
    \end{equation}
    By (Bruhat result), this is at most $|I_1|! \cdots |I_r|!\ $.

\end{proof}
Observe that Theorem~\ref{t:main} guarantees no nontrivial lower bound for $R(x)$ and gives an upper bound which is sometimes tight.
Indeed the ratio 
\begin{equation*}
    \frac{x_{1,2} x_{2,1}}{x_{1,1}x_{2,2}}
\end{equation*}
attains all values in the open interval $(0,1)$ as $x$ varies over matrices in $\tp 2$. On the other hand, special cases of the ratios in Theorem~\ref{t:main} can be shown to have both upper and nontrivial lower bounds.
\begin{cor}\label{cor:bothbounds}
For ratio $R(x)$ and matrices $C$, $D$ defined as in Theorem~\ref{t:main},
if $C = D$, then $R(x)$ is bounded above and below by
\begin{equation}\label{eq:corupperbounds}
    \frac{1}{|J_1|! \cdots |J_q|!} \leq R(x) \leq |I_1|! \cdots |I_r|! \;,
\end{equation}
for $x \in \tp n$.
\end{cor}
For example, consider the ratio
\begin{equation}
    R(x) = \frac{\perm(x_{12,34})\perm(x_{34,12})}
    {x_{1,3}x_{2,4}x_{3,1}x_{4,2}}
\end{equation}
with
$|I_1| = |I_2| = 2$, $|J_1| = |J_2| = |J_3| = |J_4| = 1$, and
\begin{equation}
    C = D = \begin{bmatrix} 
    0 & 0 & 1 & 0 \\ 
    0 & 0 & 0 & 1 \\
    1 & 0 & 0 & 0 \\
    0 & 1 & 0 & 0
    \end{bmatrix}.
\end{equation}
By Corollary~\ref{cor:bothbounds}, we have
\begin{equation*}
    \frac{1}{1!^4} \leq R(x) \leq 2!^2.
    \end{equation*}
    It is easy to see that $R(x)$ attains values arbitrarily close to $4$ as $x$ approaches the matrix of all ones.
It is also possible to show that $R(x)$
attains values arbitrarily close to $1$.  Indeed, consider the matrix $A = A(\epsilon) = (a_{i,j})$ defined by
\begin{equation}
    A(\epsilon) = \begin{bmatrix}
        1 & 1 & \epsilon & \epsilon^3 \\
    1 & 2 & 1 & \epsilon \\
    \epsilon & 1 & 2 & 1 \\
    \epsilon^3 & \epsilon & 1 & 1
    \end{bmatrix},
\end{equation}
where $\epsilon$ is positive and close to $0$.
To see that $A(\epsilon)$ is totally positive, 
it suffices to verify the positivity of the
sixteen minors $\det(A_{[a_1,b_1],[a_2,b_2]})$ indexed by pairs of 
intervals, at least one of which contains $1$~\cite[Thm.\,9]{FominTPTest}.
Observe that we have
$a_{1,j} > 0$ and $a_{i,1} > 0$ for all $i,j$.  Also,
\begin{equation*}
\begin{aligned}
    &\det(A_{12,12}) = 1,\\
    &\det(A_{12,23}) = \det(A_{23,12}) = 1 - 2\epsilon,\\
    &\det(A_{12,34}) = \det(A_{34,12}) = \epsilon^2 - \epsilon^3,\\
    &\det(A_{123,123}) = 1 + 2 \epsilon - 2 \epsilon^2,\\
    &\det(A_{123,234}) = \det(A_{234,123}) = 1 - 4\epsilon + \epsilon^2 + 3\epsilon^3,\\
    &\det(A) = 
4\epsilon-6\epsilon^2-2\epsilon^3+9\epsilon^4-2\epsilon^5-3\epsilon^6.    
    \end{aligned}
\end{equation*}
It follows that 
we have
\begin{equation*}
    \lim_{\epsilon \rightarrow 0^+} R(A) = 
    \lim_{\epsilon \rightarrow 0^+}
    \frac{(\epsilon^2 - \epsilon^3)^2}{\epsilon^4} = 
\lim_{\epsilon \rightarrow 0^+} 1 - 2\epsilon + \epsilon^2 = 1.
    \end{equation*}
    

In the case that all submatrices in (\ref{eq:mainratiox}) are principal, the necessary condition (\ref{eq:ST0}) for boundedness is in fact sufficient to guarantee the existence of upper and nontrivial lower bounds.
\begin{cor}\label{cor:principalbounds}
For ratio $R$ as in Theorem~\ref{t:main},
if all submatrices in $R(x)$ are principal, ($I_k = I'_k$, $J_k = J'_k$ for all $k$), then
$R(x)$ is bounded above and below as in (\ref{eq:corupperbounds}).
\end{cor}
\begin{proof}
For principal submatrices $x_{I_1,I_1},\dotsc$,
    the condition (\ref{eq:ST0}) implies the equality of the matrices $C$ and $D$ in (\ref{eq:CandD}):
    this matrix is diagonal with $(i,i)$ entry equal to the multiplicity of $i$ in $I_1 \multiu \cdots \multiu I_r$.
\end{proof}

For example, consider the ratio 
\begin{equation}
\frac{\perm(x_{I,I})\;\perm(x_{J,J})}
{\perm(x_{I\cup J, I\cup J})\;\perm(x_{I\cap J, I\cap J})}
\end{equation}
coming from the (false) permanental version (\ref{eq:kotelperm}) of Koteljanskii's inequality (\ref{eq:koteljanskii}).
By Corollary~\ref{cor:principalbounds}, the four principal submatrices of $x$ imply that the exponent matrices $C$ and $D$ are equal and diagonal with $(i,i)$ entry equal to the multiplicity of $i$ in $I \multiu J$. Thus Corollary~\ref{cor:bothbounds} gives the lower and upper bounds
\begin{equation}
    \frac 1{|I\cup J|!\; |I \cap J |!}, \qquad
    |I|! \; |J|!
\end{equation}
as claimed in (\ref{eq:kotelratio}).  These bounds are not in general tight. Consider the special case
\begin{equation}\label{eq:worsebounds}
\frac{1}{3!1!}\leq
\frac{\perm(x_{12,12})\perm(x_{23,23})}{\perm(x_{123,123})\perm(x_{2,2})}
\leq (2!)^2
\end{equation}
with 
\begin{equation}
    C = D = \begin{bmatrix}
        1 & 0 & 0 \\ 0 & 2 & 0 \\ 0 & 0 & 1 
    \end{bmatrix}.
\end{equation}
We improve (\ref{eq:worsebounds}) as follows.


\begin{prop}
For $x \in \tp 3$ we have
\begin{equation}\label{eq:half2}\frac{1}{2}\leq\frac{\perm(x_{12,12})\perm(x_{23,23})}{\perm(x_{123,123})\perm(x_{2,2})}\leq 2.
\end{equation}
\end{prop}
\begin{proof}
 The first inequality follows from expanding
 $$2\cdot\perm(x_{12,12})\perm(x_{23,23})-\perm(x_{123,123})\perm(x_{2,2})$$ and grouping terms as
\begin{equation*}
    \begin{aligned}
    &(x_{11}x^2_{22}x_{33}-x_{13}x^2_{22}x_{31}) +(x_{12}x_{21}x_{22}x_{33}-x_{12}x_{22}x_{23}x_{31})
    + (x_{12}x_{21}x_{23}x_{32}-x_{13}x_{21}x_{22}x_{32})\\
    &\quad + x_{11}x_{22}x_{23}x_{32}+x_{12}x_{21}x_{23}x_{32}\\
&= \det(x_{13,13})x^2_{22}+\det(x_{23,13})x_{12}x_{22}+\det(x_{12,23})x_{21}x_{32}+x_{11}x_{22}x_{23}x_{32}+x_{12}x_{21}x_{23}x_{32}.
\end{aligned}
\end{equation*}
Similarly, the second inequality follows from expanding
$$2\cdot\perm(x_{123,123})\perm(x_{2,2})-\perm(x_{12,12})\perm(x_{23,23})$$
and grouping terms as
\begin{equation*}
x_{11}x^2_{22}x_{33}+\det(x_{23,23})x_{12}x_{21}+2x_{12}x_{22}x_{23}x_{31}+x_{11}x_{22}x_{23}x_{32}+2x_{13}x_{21}x_{22}x_{32}+2x_{13}x^2_{22}x_{31}.
\end{equation*}
\end{proof}



The authors believe that even these bounds are not tight.  The smallest and greatest values for (\ref{eq:half2}) that we have found are $2/3$ and $121/114$, respectively.

\section{Future directions}\label{s:open}
It would be interesting to characterize the ratios (\ref{eq:R}) which are bounded by $1$, i.e., to solve the following problem.
\bp 
Characterize the differences
\begin{equation}\label{eq:perprodminusperprod}
\perm(x_{J_1,J'_1}) \cdots \perm(x_{J_q,J'_q}) - 
    \perm(x_{I_1,I'_1}) \cdots \perm(x_{I_r,I'_r})
\end{equation}
which are totally nonnegative polynomials.
\ep
To consider a special case, it is possible to show that for small $n$, the sets $I = [2n] \ssm 2\mathbb Z$, $J = [2n] \cap 2 \mathbb Z$
define a totally nonnegative
polynomial
\begin{equation}\label{eq:sihong}
\perm(x_{[n],[n]})\; \perm(x_{[n+1,2n],[n+1,2n]}) - 
\perm(x_{I,I})\; \perm(x_{J,J}).
\end{equation}
If this polynomial is totally nonnegative in general, then
it provides a permanental analog of the known totally nonnegative polynomial
\begin{equation*}
    \det(x_{I,I}) \det(x_{J,J}) - 
    \det(x_{[n,n]})
    \det(x_{[n+1,2n],[n+1,2n]}).
\end{equation*}
%
\begin{conj}
The polynomial (\ref{eq:sihong}) is totally nonnegative for all 
$n$.
\end{conj}
Other families of possible inequalities are suggested by the known inequalities appearing in
Proposition~\ref{p:perminmax}.
In particular,
the first inequality there compares the permanent to a product of permanents of $1 \times 1$ matrices.  Comparing further to products of permanents of the form
\begin{equation}\label{eq:permprodgen}
\perm(x_{I_1,I_1}) \cdots \perm(x_{I_r,I_r}),
\end{equation}
we obtain polynomials
of the forms
\begin{equation}
    \perm(x) - \perm(x_{I_1,I_1}) \cdots \perm(x_{I_r,I_r}),
    \qquad
    \perm(x_{I_1,I_1}) \cdots \perm(x_{I_r,I_r}) - x_{1,1} \cdots x_{n,n},
\end{equation}
which are totally nonnegative because they belong to $\spn_{\mathbb N}\{ x^{e,w}
\,|\, w \in \sn \}$.
It is natural to ask if the cardinalities of the index sets determine whether a difference of the form (\ref{eq:perprodminusperprod}) is totally nonnegative, but this is not the case.
It is natural then to ask how {\em averages} of such products compare to one another.  This problem is open.  (See \cite{BJMajor}, \cite[Prob.\,5.3]{SkanSoskinBJArx}.)
\bp Characterize the pairs of partitions $\lambda = (\lambda_1,\dotsc,\lambda_r)$,
$\mu = (\mu_1,\dotsc,\mu_q)$ such that
\begin{equation}\label{eq:bjperm}
\sumsb{(I_1,\dotsc,I_r)\\|I_k| = \lambda_k}
\frac{\perm(x_{I_1,I_1}) \cdots \perm(x_{I_r,I_r})}{\tbinom{n}{\lambda_1,\dotsc,\lambda_r}}
-
\sumsb{(J_1,\dotsc,J_q)\\|J_k| = \mu_k}
\frac{\perm(x_{J_1,J_1}) \cdots \perm(x_{J_q,J_q})}{\tbinom{n}{\mu_1,\dotsc,\mu_r}}
\end{equation}
is totally nonnegative.
\ep

Now consider generalizing the second inequality in Proposition \ref{p:perminmax} to products of permanents of the form (\ref{eq:permprodgen}). Differences of the form
\begin{equation}\label{eq:nottnn}
\frac{\perm(x_{I_1,I_1}) \cdots \perm(x_{I_r,I_r})}{|I_1|! \cdots |I_r|!} - 
\frac{\perm(x)}{n!}
\end{equation}
are {\em not} totally nonnegative, while differences of the form
\begin{equation}
x_{1,1} \cdots x_{n,n} - \frac{\perm(x_{I_1,I_1}) \cdots \perm(x_{I_r,I_r})}{|I_1|! \cdots |I_r|!}
\end{equation}
are (by Proposition~\ref{p:perminmax}).  It is natural then to ask about the averages of differences (\ref{eq:nottnn}), over all set partitions $(I_1,\dotsc,I_r)$ of a partition $\lambda$.
\bp\label{p:avgminusperm}
Decide if for fixed $\lambda = (\lambda_1,\dotsc,\lambda_r)$, the polynomial
\begin{equation*}
    \sumsb{(I_1,\dotsc,I_r)\\|I_k| = \lambda_k} \perm(x_{I_1,I_1}) \cdots \perm(x_{I_r,I_r})
    -
    \perm(x)
\end{equation*}
is totally nonnegative.
\ep


To illustrate (\ref{eq:nottnn}) and Problem~\ref{p:avgminusperm}, let us consider the case that $n = 3$. It is straightforward to show that
\begin{equation}\label{eq:121233}
\frac{\perm(x_{12,12}) x_{3,3}}{2!1!} - \frac{\perm(x)}{3!},
\end{equation}
equivalently, $3\perm(x_{12,12})x_{3,3} - \perm(x)$, is totally nonnegative because the latter expression equals a sum of matrix minors.
Similarly,
\begin{equation}\label{eq:232311}
\frac{\perm(x_{23,23}) x_{1,1}}{2!1!} - \frac{\perm(x)}{3!},
\end{equation}
is totally nonnegative.
On the other hand,
\begin{equation}\label{eq:131322}
\frac{\perm(x_{13,13}) x_{2,2}}{2!1!} - \frac{\perm(x)}{3!},
\end{equation}
is not, because its evaluation at
\begin{equation*}
 \begin{bmatrix}
        1 & 1 & 0 \\
        1 & 1 & 1 \\
        1 & 1 & 1 
    \end{bmatrix}
\end{equation*}
is negative.
On the other hand, two times
the sum of the three differences (\ref{eq:121233}) -- (\ref{eq:131322}) is
\begin{equation*}
 2x_{1,1}x_{2,2}x_{3,3} - x_{1,2}x_{2,3}x_{3,1} - x_{1,3}x_{2,1}x_{3,2},
 \end{equation*}
 which is totally nonnegative by Theorem~\ref{t:DGS}.


\bibliography{my}
\end{document}